\documentclass[12pt]{amsart}
\theoremstyle{plain}
\newtheorem{theorem}{Theorem}[section]
\newtheorem{corollary}[theorem]{Corollary}
\newtheorem{example}[theorem]{Example}

\theoremstyle{definition}

\theoremstyle{remark}
\newtheorem{remark}[theorem]{Remark}
\newtheorem{problem}[theorem]{Problem}

\begin{document}

\title[Orthogonality-preserving and ($C^*$-)conformal module mappings]{Orthogonality-preserving, $C^*$-conformal and conformal module mappings on Hilbert $C^*$-modules}
\author[M.~Frank]{Michael Frank}
\address{Hochschule f\"ur Technik, Wirtschaft und Kultur (HTWK) Leipzig, Fachbereich
         Informatik, Mathematik und Naturwissenschaften, PF 301166, 04251 Leipzig, Germany}
\email{mfrank@imn.htwk-leipzig.de}
\thanks{The research has been supported by a grant of Deutsche Forschungsgemeinschaft (DFG) and by the RFBR-grant 07-01-91555.}
\author[A.~S.~Mishchenko]{Alexander S.~Mishchenko}
\address{Moscow State University, Faculty of Mechanics and Mathematics, Main Building,
         Leninskije Gory, 119899 Moscow, Russia}
\email{asmish@mech.mat.msu.su, asmish.math@gmail.com}
\author[A.~A.~Pavlov]{Alexander A.~Pavlov}
\address{Moscow State University, 119 922 Moscow, Russia, and
Universit\`a degli Studi di Trieste, I-34127 Trieste, Italy}
\email{axpavlov@mail.ru}

\begin{abstract}
We investigate orthonormality-preserving, $C^*$-conformal and conformal 
mo\-dule mappings on full Hilbert $C^*$-modules to obtain their general structure.
Orthogonality-preserving bounded module maps $T$ act as a multiplication
by an element $\lambda$ of the center of the multiplier algebra of the
$C^*$-algebra of coefficients combined with an isometric module operator
as long as some polar decomposition conditions for the specific element
$\lambda$ are fulfilled inside that multiplier algebra. Generally,
$T$ always fulfils the equality $\langle T(x),T(y) \rangle =
| \lambda |^2 \langle x,y \rangle$ for any elements $x,y$ of the
Hilbert $C^*$-module. At the contrary, $C^*$-conformal and conformal 
bounded module maps are shown to be only the positive real 
multiples of isometric module operators.
\end{abstract}
\maketitle

The set of all orthogonality-preserving bounded linear mappings on
Hilbert spaces is fairly easy to describe, and it coincides with the
set of all conformal linear mappings there: a linear map $T$ between
two Hilbert spaces $H_1$ and $H_2$ is orthogonality-preserving if and 
only if $T$ is the scalar multiple of an isometry $V$ with 
$V^*V={\rm id}_{H_1}$. Furthermore, the
set of all orthogonality-preserving mappings $\{ \lambda \cdot V:
\lambda \in {\mathbb C}, V^*V={\rm id}_{H_1} \}$ corresponds to 
the set of all those maps which transfer tight frames of $H_1$ 
into tight frames of (norm-closed) subspaces $V(H_1)$ of $H_2$, cf.~\cite{HL2000}.

The latter fact transfers to the more general situation of standard tight frames of
Hilbert $C^*$-modules in case the image submodule is an orthogonal
summand of the target Hilbert $C^*$-module, cf.~\cite[Prop.~5.10]{FL2002}.
Also, module isometries of Hilbert $C^*$-modules are always
induced by module unitary operators between them, \cite{L1994},
\cite[Prop.~2.3]{IT2008}.
However, in case of a non-trivial center of the multiplier algebra of 
the $C^*$-algebra of coefficients the property of a bounded module map to be
merely  orthogonality-preserving might not infer the property of that
map to be ($C^*$-)\-conformal or even isometric. So the goal of the present
note is to derive the structure of arbitrary orthogonality-preserving,
$C^*$-conformal or conformal bounded module mappings on Hilbert $C^*$-modules 
over (non-)unital $C^*$-algebras without any further assumption.

Partial solutions can be found in a publication by
D.~Ili{\v{s}}evi\'c and A.~Turn{\v{s}}ek for $C^*$-algebras
$A$ of coefficients which admit a faithful $*$-representation $\pi$
on some Hilbert space $H$ such that $K(H) \subseteq \pi(A) \subseteq
B(H)$, cf.~\cite[Thm.~3.1]{IT2008}.
Orthogonality-preserving mappings have been mentioned also in
a paper by J.~Chmieli\'nski, D.~Ili{\v{s}}evi\'c, M.~S.~Moslehian,
Gh.~Sadeghi, \cite[Th.~2.2]{ChIMS2008}.
In two working drafts \cite{LKW09-1,LKW09-2} by Chi-Wai Leung, Chi-Keung Ng
and Ngai-Ching Wong found by a Google search in May 2009 we obtained
further partial results on orthogonality-preserving linear
mappings on Hilbert $C^*$-modules.

Orthogonality-preserving bounded linear mappings between
$C^*$-algebras  have been considered by J.~Schweizer in his
Habilitation Thesis in 1996, \cite[Prop.~4.5-4.8]{Schw1996}.
His results are of interest in application to the
linking $C^*$-algebras of Hilbert $C^*$-modules.

A bounded module map $T$ on a Hilbert $C^*$-module $\mathcal M$ is
said to be {\it orthogonality-preserving} if $\langle T(x),T(y)
\rangle = 0$ in case $\langle x,y \rangle = 0$ for certain $x,y \in
\mathcal M$. In particular, for two Hilbert $C^*$-modules $\mathcal M$,
$\mathcal N$ over some $C^*$-algebra $A$ a bounded module map
$T: {\mathcal M} \to {\mathcal N}$ is orthogonality-preserving if and
only if the validity of the inequality $\langle x,x \rangle \leq
\langle x+ay,x+ay \rangle$ for some $x,y \in \mathcal M$ and any
$a \in A$ forces the validity of the inequality $\langle T(x),T(x)
\rangle \leq \langle T(x)+aT(y), T(x)+aT(y) \rangle$ for any $a \in A$,
cf.~\cite[Cor.~2.2]{IT2008}.
So the property of a bounded module map to be orthogonality
preserving has a geometrical meaning considering pairwise
orthogonal one-dimensional $C^*$-submodules and their orthogonality
in a geometric sense.

Orthogonality of elements of Hilbert $C^*$-modules with respect to
their $C^*$-valued inner products is different from the classical
James-Birkhoff orthogonality defined with respect to the norm derived
from the $C^*$-valued inner products, in general. Nevertheless, the
results are similar in both situations, and the roots of both these
problem fields coincide for the particular situation of Hilbert
spaces. For results in this parallel direction the reader might
consult publications by A.~Koldobsky \cite{Ko1993}, by
A.~Turn{\v{s}}ek \cite{Tu2005}, by J.~Chmieli\'nsky
\cite{Ch2005-2,Ch2007}, and by A.~Blanco and A.~Turn{\v{s}}ek
\cite{BlTu2006}, among others.

Further resorting to {\it $C^*$-conformal} or {\it conformal}
mappings on Hilbert $C^*$-modules, i.e. bounded module maps
preserving either a generalized $C^*$-valued angle $\langle x,y
\rangle / \|x\|\|y\|$ for any $x,y$ of the Hilbert $C^*$-module or
its normed value, we consider a particular situation of
orthogonality-preserving mappings. Surprisingly, both these sets
of orthogonality-preserving and of ($C^*$-)conformal mappings are
found to be different in case of a non-trivial center of the
multiplier algebra of the underlying $C^*$-algebra of
coefficients.

The content of the present paper is organized as follows:
In the following section we investigate the general structure of
orthogonality-preserving bounded module mappings on Hilbert
$C^*$-modules. The results are formulated in Theorem \ref{main}
and Theorem \ref{thm02}. In the last section we characterize
$C^*$-conformal and conformal bounded module mappings on Hilbert
$C^*$-modules, see Theorem \ref{teo:$C^*$-conformal_maps}
and Theorem \ref{teo:conformal_maps}.

Since we rely only on the very basics of $*$-representation
and duality theory of $C^*$-algebras and of Hilbert $C^*$-module
theory, respectively, we refer the reader to the monographs by
M.~Takesaki \cite{Takesaki} and by V.~M.~Manuilov and E.~V.~Troitsky
\cite{MaTrobook}, or to other relevant monographical publications
for basic facts and methods of both these theories.


\section{Orthogonality-preserving mappings}

\medskip
The set of all orthogonality-preserving bounded linear mappings on
Hilbert spaces is fairly easy to describe. For a given Hilbert
space $H$ it consists of all scalar multiples of isometries $V$,
where an isometry is a map $V: H \to H$ such that  $V^*V = {\rm
id}_H$. Any bounded linear orthogonality-preserving map $T$
induces a bounded linear map $T^*T: H \to H$. For a non-zero
element $x \in H$ set $T^*T(x)= \lambda_x x + z$ with $z \in
\{x\}^\bot$ and $\lambda_x \in {\mathbb C}$. Then the given relation
$\langle x,z \rangle = 0$ induces the equality
\[
   0 = \langle T(x),T(z)\rangle = \langle T^*T(x),z \rangle =
   \langle \lambda_x x + z, z \rangle = \langle z,z \rangle \, .
\]
Therefore, $z=0$ by the non-degeneratedness of the inner product,
and $\lambda_x \geq 0$ by the positivity of $T^*T$. Furthermore,
for two orthogonal elements $x,y \in H$ one has the equality
\[
   \lambda_{x+y} (x+y) = T^*T(x+y) = \lambda_x x + \lambda_y y
\]
which induces the equality $\lambda_{x+y} \langle x,x \rangle =
\lambda_x \langle x,x \rangle$ after scalar multiplication by $x
\in H$. Since the element $\langle x,x \rangle$ is invertible in
$\mathbb C$ we can conclude that the orthogonality-preserving
operator $T$ induces an operator $T^*T$ which acts as a positive
scalar multiple $\lambda \cdot {\rm id}_H$ of the identity operator on
any orthonormal basis of the Hilbert space $H$. So $T^*T = \lambda
\cdot {\rm id}_H$ on the Hilbert space $H$ by linear continuation. The
polar decomposition of $T$ inside the von Neumann algebra $B(H)$
of all bounded linear operators on $H$ gives us the equality $T =
\sqrt{\lambda} V$ for an isometry $V: H \to H$, i.e.~with
$V^*V={\rm id}_H$. The positive number $\sqrt{\lambda}$ can be
replaced by an arbitrary complex number of the same modulus
multiplying by a unitary $u \in {\mathbb C}$. In this case the
isometry $V$ has to be replaced by the isometry $u^*V$ to yield
another decomposition of $T$ in a more general form.

\medskip
As a natural generalization of the described situation one may
change the algebra of coefficients to arbitrary $C^*$-algebras $A$
and the Hilbert spaces to $C^*$-valued inner product $A$-modules,
the (pre-)Hilbert $C^*$-modules. Hilbert $C^*$-modules are an
often used tool in the study of locally compact quantum groups and
their representations, in noncommutative geometry, in $KK$-theory,
and in the study of completely positive maps between
$C^*$-algebras, among other research fields.

To be more precise, a (left) {\it pre-Hilbert $C^*$-module} over a
(not necessarily unital) $C^*$-algebra $A$ is a left $A$-module
$\mathcal M$ equipped with an $A$-valued inner product $\langle
\cdot , \cdot \rangle : {\mathcal M} \times {\mathcal M} \to A$,
which is $A$-linear in the first variable and has the properties
$\langle x,y \rangle=\langle y,x \rangle ^*$, $\langle x,x \rangle
\geq 0$ with equality if and only if $x=0$. We always suppose that
the linear structures of $A$ and $\mathcal M$ are compatible. A
pre-Hilbert $A$-module $\mathcal M$ is called a {\it Hilbert
$A$-module} if $\mathcal M$ is a Banach space with respect to the
norm $\| x \|= \| \langle x,x\rangle \|^{1/2}$.

Consider bounded module orthogonality-preserving maps $T$ on
Hilbert $C^*$-modules $\mathcal M$. For several reasons we cannot
repeat the simple arguments given for Hilbert spaces in the
situation of an arbitrary Hilbert $C^*$-module, in general. First
of all, the bounded module operator $T$ might not admit a bounded
module operator $T^*$ as its adjoint operator, i.e. satisfying the
equality $\langle T(x),y \rangle = \langle x,T^*(y) \rangle$ for
any $x,y \in \mathcal M$. Secondly, orthogonal complements of
subsets of a Hilbert $C^*$-module might not be orthogonal direct
summands of it. Last but not least, Hilbert $C^*$-modules might
not admit analogs (in a wide sense) of orthogonal bases. So the
understanding of the nature of bounded module
orthogonality-preserving operators on Hilbert $C^*$-modules
involves both more global and other kinds of localisation
arguments.

\begin{example} {\rm
Let $A$ be the $C^*$-algebra of continuous functions on the unit
interval $[0,1]$ equipped with the usual Borel topology. Let $I =
C_0((0,1])$ be the $C^*$-subalgebra of all continuous functions on
$[0,1]$ vanishing at zero. $I$ is a norm-closed two-sided ideal of
$A$.
\newline
Let ${\mathcal M}_1 = A \oplus A$ be the Hilbert $A$-module that
consists of two copies of $A$, equipped with the standard
$A$-valued inner product on it. Consider the multiplication $T_1$
of both parts of ${\mathcal M}_1$ by the function $a(t) \in A$, $a(t)
:= t$ for any $t \in [0,1]$. Obviously, the map $T_1$ is bounded,
$A$-linear, injective and orthogonality-preserving. However, its range
is even not norm-closed in ${\mathcal M}_1$. 
\newline
Let ${\mathcal M}_2 = I \oplus l_2(A)$ be the orthogonal direct
sum of a proper ideal $I$ of $A$ and of the standard countably generated Hilbert
$A$-module $l_2(A)$. Consider the shift operator $T_2: {\mathcal M}_2
\to {\mathcal M}_2$ defined by the formula $T_2((i,a_1,a_2,...))=
(0,i,a_1,a_2,...)$ for $a_k \in A$, $i \in I$. It is an isometric
$A$-linear embedding of ${\mathcal M}_2$ into itself and, hence,
orthogonality-preserving, however $T_2$ is not adjointable.
}
\end{example}

To formulate the result on orthogonality-preserving mappings
we need a construction by W.~L.~Paschke (\cite{Pa1973}):
for any Hilbert $A$-module $\mathcal M$ over any $C^*$-algebra
$A$ one can extend $\mathcal M$ canonically to a Hilbert
$A^{**}$-module ${\mathcal M}^{\#}$ over the bidual Banach space
and von Neumann algebra $A^{**}$ of $A$
\cite[Th.~3.2, Prop.~3.8, {\S }4]{Pa1973}. For this aim the
$A^{**}$-valued pre-inner product can be defined by the formula
\[
[a\otimes x,b\otimes y]=a\langle x,y\rangle b^*,
\]
for elementary tensors of $A^{**}\otimes M$, where $a,b\in A^{**}$,
$x,y \in M$. The quotient module of $A^{**}\otimes M$ by the
set of all isotropic vectors is denoted by ${\mathcal M}^{\#}$. It
can be canonically completed to a self-dual Hilbert $A^{**}$-module
$\mathcal N$ which is isometrically algebraically isomorphic to 
the $A^{**}$-dual $A^{**}$-module of ${\mathcal M}^{\#}$. 
$\mathcal N$ is a dual Banach space itself,
(cf.~\cite[Thm.~3.2, Prop.~3.8, {\S}4]{Pa1973}.)
Every $A$-linear bounded map $T: {\mathcal M} \to {\mathcal M}$ 
can be continued to a unique $A^{**}$-linear map $T: {\mathcal M}^{\#} 
\to {\mathcal M}^{\#}$ preserving the operator norm and obeying the 
canonical embedding $\pi'({\mathcal M})$ of
$\mathcal M$ into ${\mathcal M}^{\#}$. Similarly, $T$ can be
further extended to the self-dual Hilbert $A^{**}$-module $\mathcal N$.
The extension is such that the isometrically algebraically embedded
copy $\pi'({\mathcal M})$ of $\mathcal M$ in $\mathcal N$ is a
w*-dense $A$-submodule of $\mathcal N$, and that $A$-valued inner
product values of elements of $\mathcal M$ embedded in $\mathcal N$
are preserved with respect to the $A^{**}$-valued inner product on
$\mathcal N$ and to the canonical isometric embedding $\pi$ of $A$
into its bidual Banach space $A^{**}$. Any bounded $A$-linear
operator $T$ on $\mathcal M$ extends to a unique bounded
$A^{**}$-linear operator on $\mathcal N$ preserving the operator
norm, cf. \cite[Prop.~3.6, Cor.~3.7, {\S}4]{Pa1973}. The extension
of bounded $A$-linear operators from $\mathcal M$ to $\mathcal N$
is continuous with respect to the w*-topology on ${\mathcal N}$.
For topological characterizations of self-duality of Hilbert
$C^*$-modules over $W^*$-algebras we refer to \cite{Pa1973}, 
\cite[Thm.~3.2]{Frank1990} and to \cite{Schw1996,Schw2002}: a
Hilbert $C^*$-module $\mathcal K$ over a $W^*$-algebra $B$ is self-dual, 
if and only if its unit ball is complete with respect to the topology 
induced by the semi-norms $\{ |f(\langle .,x \rangle)| : x \in 
{\mathcal K}, f \in B^*, \|x\| \leq 1, \|f\| \leq 1 \}$, 
if and only if its unit ball is complete with respect to the topology 
induced by the semi-norms $\{ f(\langle .,. \rangle)^{1/2} : 
f \in B^*, \|x\| \leq 1, \|f\| \leq 1\}$. The first topology coincides 
with the $w^*$-topology on $\mathcal K$ in that case. 

Note, that in the construction above $\mathcal M$ is always $w^*$-dense
in $\mathcal N$, as well as for any subset of $\mathcal M$ the
respective construction is $w^*$-dense in its biorthogonal
complement with respect to $\mathcal N$. However, starting
with a subset of $\mathcal N$ its biorthogonal complement
with respect to $\mathcal N$ might not have a $w^*$-dense 
intersection with the embedding of $\mathcal M$ into $\mathcal N$,
cf. \cite[Prop. 3.11.9]{Ped1979}.

\begin{example}
{\rm
Let $A$ be the $C^*$-algebra of all continuous functions on the
unit interval, i.e. $A=C([0,1])$. In case we consider $A$ as a Hilbert
$C^*$-module over itself and an orthogonality-preserving map
$T_0$ defined by the multiplication by the function $a(t) =
t \cdot (\sin(1/t) + {\bf i}\cos(1/t))$ we obtain that
the operator $T_0$ cannot be written as the combination of
a multiplication by a positive element of $A$ and of an
isometric module operator $U_0$ on ${\mathcal M}=A$. The reason
for this phenomenon is the lack of a polar decomposition
of $a(t)$ inside $A$. Only a lift to the bidual von
Neumann algebra $A^{**}$ of $A$ restores the simple
description of the continued operator $T_0$ as the combination
of a multiplication by a positive element (of the center) of $A$
and an isometric module operator on ${\mathcal M}^{\#}=
{\mathcal N}=A^{**}$. The unitary part of $a(t)$ is a so-called
local multiplier of $C([0,1])$, i.e. a multiplier of $C_0((0,1])$.
But it is not a multiplier of $C([0,1])$ itself. We shall show that
this example is a very canonical one.
}
\end{example}

We are going to demonstrate the following fact on the nature of
orthogonality-preserving bounded module mappings on
Hilbert $C^*$-modules.
Without loss of generality, one may assume that the range of the
$A$-valued inner product on $\mathcal M$ in $A$ is norm-dense in
$A$. Such Hilbert $C^*$-modules are called {\it full} Hilbert
$C^*$-modules.
Otherwise $A$ has to be replaced by the range of the $A$-valued
inner product which is always a two-sided norm-closed $*$-ideal
of $A$. The sets of all adjointable bounded module operators and of
all bounded module operators on ${\mathcal M}$, resp., are invariant
with respect to such changes of sets of coefficients of Hilbert
$C^*$-modules, cf.~\cite{Pa1973}.

\begin{theorem} \label{main}
   Let $A$ be a $C^*$-algebra, $\mathcal M$ be a full Hilbert $A$-module and
   ${\mathcal M}^{\#}$ be its canonical $A^{**}$-extension.
   Any orthogonality-preserving bounded $A$-linear operator $T$ on
   $\mathcal M$ is of the form $T = \lambda V$, where $V: {\mathcal M}^{\#}
   \to {\mathcal M}^{\#}$ is an isometric $A$-linear embedding and
   $\lambda$ is a positive element of the centre $Z(M(A))$ of the
   multiplier algebra $M(A)$ of $A$.
   \newline
   If any element $\lambda' \in Z(M(A))$ with $|\lambda'|=\lambda$
   admits a polar decomposition inside $Z(M(A))$ then the operator $V$
   preserves $\pi'({\mathcal M}) \subset {\mathcal M}^{\#}$. So $T =
   \lambda \cdot V$ on ${\mathcal M}$.
\end{theorem}

In \cite[Thm.~3.1]{IT2008} D.~Ili{\v{s}}evi\'c and
A.~Turn{\v{s}}ek proved Theorem \ref{main} for the particular case
if for some Hilbert space $H$ the $C^*$-algebra $A$ admits an
isometric representation $\pi$ on $H$ with the property $K(H)
\subset \pi(A) \subset B(H)$. In this situation $Z(M(A)) = 
{\mathbb C}$.

\begin{proof}
We want to make use of the canonical nondegenerate isometric
$*$-representation $\pi$ of a $C^*$-algebra $A$ in its bidual
Banach space and von Neumann algebra $A^{**}$ of $A$, as well
as of its extension $\pi': {\mathcal M} \to {\mathcal M}^{\#}
\to {\mathcal N}$ and of its operator extension. That is, we
switch from the triple $\{ A,{\mathcal M}, T \}$ to the triple
$\{ A^{**},{\mathcal M}^{\#} \subseteq {\mathcal N}, T \}$.

We have to demonstrate that for orthogonality-preserving bounded
$A$-linear mappings $T$ on $\mathcal M$ the respective extended
bounded $A^{**}$-linear operator on $\mathcal N$ is still
orthogonality-preserving for $\mathcal N$. 
Let $x$ be an element of $\mathcal N$ and denote by $\mathcal K$ 
its biorthogonal complement of with respect to $\mathcal N$. 
Then $\mathcal K$ is a direct orthogonal summand of $\mathcal N$ 
because $\mathcal N$ and $\mathcal K$ are self-dual Hilbert 
$A^{**}$-modules. Consider any positive normal state $f$ on $A^{**}$ 
with $f(\langle x,x \rangle) \not= 0$. Since the $A$-valued inner 
product $\langle .,. \rangle$ on $\mathcal M$ continues to an 
$A^{**}$-valued inner product $\langle .,. \rangle$ on $\mathcal N$ 
in a unique way by \cite[Thm.~3.2]{Pa1973}, the possibly degenerated 
complex-valued inner product $f(\langle .,. \rangle)$ on $\mathcal M$ 
continues to a possibly degenerated complex-valued inner product 
$f(\langle .,. \rangle)$ on $\mathcal N$ in a unique way. Consider 
$x \in \mathcal N$ and its module-biorthogonal complement $\mathcal K$ 
with respect to $\mathcal N$. The intersection of $\mathcal K$ with 
the isometrically embedded copy of $\mathcal M$ in $\mathcal N$ has to 
be a weakly-dense subset of $\mathcal K$ after factorization by the 
kernel of $f(\langle .,. \rangle)^{1/2}$, otherwise the continuation 
of $f(\langle .,. \rangle)$ from ${\mathcal M} \cap {\mathcal K}$ 
to $\mathcal K$ would be non-unique. So $x$ can be represented as 
a weak limit of a Hilbert space sequence of the subset $({\mathcal K} 
\cap {\mathcal M}) / kernel(f(\langle .,.\rangle)^{1/2})$ in
${\mathcal N} / kernel(f(\langle .,.\rangle)^{1/2})$. 
Now, take another non-trivial element $y \in {\mathcal N}$ with 
$\langle x,y \rangle=0$. Then the module-biorthogonal complement 
${\mathcal L}$ of $y$ with respect to $\mathcal N$ is orthogonal to 
${\mathcal K}$. Repeat the construction for $y$ fixing $f$. 
Since $f(\langle z,t \rangle)=0$ for any $z \in ({\mathcal K} \cap 
{\mathcal M}) / kernel(f(\langle .,.\rangle)^{1/2})$ and any $t \in 
({\mathcal L} \cap {\mathcal M}) / kernel(f(\langle .,.\rangle)^{1/2})$, 
and since these sets are weakly dense in ${\mathcal K}/kernel(f(\langle .,.
\rangle)^{1/2})$ and ${\mathcal L}/kernel(f(\langle .,.\rangle)^{1/2})$, 
respectively, the weak continuity of the map $T$ and the jointly
weak continuity of inner products forces $f(\langle T(z),T(t) \rangle)=0$. 
Since $f$ has been selected arbitrarily, $\langle x,y \rangle=0$ 
for some $x,y \in \mathcal N$ forces $\langle T(x),T(y) \rangle=0$. 
Note, that the arguments are so complicated because ${\mathcal K}$
or ${\mathcal L}$ might have non-$w°*$-dense intersections 
with ${\mathcal M} \subseteq {\mathcal N}$ by 
\cite[Prop.~3.11.9]{Ped1979}.

Next, we want to consider only {\it discrete} W*-algebras,
i.e.~W*-algebras for which the supremum of all minimal projections
contained in them equals their identity. (We prefer to use the word
discrete instead of atomic.) To connect to the general $C^*$-case
we make use of a theorem by Ch.~A.~Akemann stating that the
$*$-homomorphism of a $C^*$-algebra $A$ into the discrete part of
its bidual von Neumann algebra $A^{**}$ which arises as the
composition of the canonical embedding $\pi$ of $A$ into $A^{**}$
followed by the projection $\rho$ to the discrete part of $A^{**}$
is an injective $*$-homomorphism, \cite[p.~278]{Ak1969} and
\cite[p.~I]{Ak1971}. The injective $*$-homomorphism $\rho$ is partially
implemented by a central projection $p \in Z(A^{**})$ in such a
way that $A^{**}$ multiplied by $p$ gives the discrete part of
$A^{**}$. Applying this approach to our situation we reduce the
problem further by investigating the triple $\{ pA^{**},
p{\mathcal N}, pT \}$ instead of the triple $\{ A^{**}, {\mathcal
N}, T \}$, where we rely on the injectivity of the algebraic
embeddings $\rho \circ \pi: A \to pA^{**}$ and $\rho' \circ \pi':
{\mathcal M} \to p{\mathcal N}$.
The latter map is injective since $\langle x,x \rangle \not= 0$
forces $\langle px,px \rangle = p\langle x,x \rangle =
\rho \circ \pi (\langle x,x \rangle) \not= 0$.
Obviously, the bounded $pA^{**}$-linear operator $pT$ is
orthogonality-preserving for the self-dual Hilbert
$pA^{**}$-module $p{\mathcal N}$ because the orthogonal
projection of $\mathcal N$ onto $p{\mathcal N}$ and the
operator $T$ commute, and both they are orthogonality-preserving.

In the sequel we have to consider the multiplier algebra $M(A)$
and the left multiplier algebra $LM(A)$ of the $C^*$-algebra $A$.
By \cite{Ped1979} every non-degenerate injective $*$-representation
of $A$ in a von Neumann algebra $B$ extends to an injective
$*$-representation of the multiplier algebra $M(A)$ in $B$ and
to an isometric algebraic representation of the left multiplier
algebra $LM(A)$ of $A$ preserving the strict and the left strict
topologies on $M(A)$ and on $LM(A)$, respectively. In particular,
the injective $*$-representation $\rho \circ \phi$ extends to $M(A)$
and to $LM(A)$ in such a way that
\begin{eqnarray*}
    \rho \circ \phi(M(A)) & = & \{ b \in pA^{**}: b \rho \circ
    \phi(a) \in A, \, \rho \circ \phi (a) b \in A \,\, {\rm for}
    \, {\rm every} \,\, a \in A \} \, , \\
    \rho \circ \phi(LM(A)) & = & \{ b \in pA^{**}: b \rho \circ
    \phi(a) \in A \,\, {\rm for} \, {\rm every} \,\, a \in A \} \, .
\end{eqnarray*}
Since $Z(LM(A)) = Z(M(A))$ for the multiplier algebra
of $A$ of every $C^*$-algebra $A$, we have the description
\begin{eqnarray*}
    \rho \circ \phi(Z(M(A))) & = & \{ b \in pA^{**}: b \rho \circ
    \phi(a) = \rho \circ \phi (a) b \in A \,\, {\rm for}
    \, {\rm every} \,\, a \in A \} \, .
\end{eqnarray*}

Since the von Neumann algebra $pA^{**}$ is discrete the identity
$p$ can be represented as the w*-sum of a maximal set of pairwise
orthogonal atomic projections $\{ q_\alpha : \alpha \in I\}$ of
the centre $Z(pA^{**})$ of $pA^{**}$. Note, that $\sum_{\alpha
\in I} q_\alpha = p$. Select a single atomic projection $q_\alpha
\in Z(pA^{**})$ of this collection and consider the part $\{
q_\alpha  pA^{**}, q_\alpha p{\mathcal N}, q_\alpha pT \}$ of the
problem for every single $\alpha \in I$.

By \cite[Thm.~3.1]{IT2008} the operator $q_\alpha pT$ can be
described as a non-negative constant $\lambda_{q_\alpha}$ multiplied by
an isometry $V_{q_\alpha}$ on the Hilbert $q_\alpha pA^{**}$-module
$q_\alpha p{\mathcal N}$, where the isometry $V_{q_\alpha}$ preserves
the $q_\alpha pA$-submodule $q_\alpha p{\mathcal M}$ inside
$q_\alpha p{\mathcal N}$ since the operator $q_\alpha pT$ preserves
it, and multiplication by a positive number does not change this fact.
In case $\lambda_{q_\alpha}=0$ we set simply $V_{q_\alpha}=0$.

We have to show the existence of global operators on the Hilbert
$pA^{**}$-module $p{\mathcal N}$ build as w*-limits of nets of
finite sums with pairwise distinct summands of the sets
$\{ \lambda_{q_\alpha} q_\alpha : \alpha \in I\}$ and
$\{ q_\alpha V_{q_\alpha} : \alpha \in I \}$, respectively.
Additionally, we have to establish key properties of them. First,
note that the collection of all finite sums with pairwise distinct
summands of  $\{ \lambda_{q_\alpha} q_\alpha : \alpha \in I \}$
form an increasingly directed net of positive elements of the centre
of the operator algebra ${\rm End}_{pA^{**}}(p{\mathcal N})$, which is
$*$-isomorphic to the von Neumann algebra $Z(pA^{**})$. This net is
bounded by $\|pT\| \cdot {\rm id}_{p{\mathcal N}}$ since the operator
$pT$ admits an adjoint operator on the self-dual Hilbert $pA^{**}$-module
$p{\mathcal N}$ by \cite[Prop. 3.4]{Pa1973} and since for any finite
subset $I_0$ of $I$ the inequality
\[
 0 \leq \sum_{\alpha \in I_0} \lambda_{q_\alpha}^2 \cdot
        {\rm id}_{q_\alpha p {\mathcal N}}
     =  \sum_{\alpha \in I_0} q_\alpha pT^*T
   \leq pT^*T
   \leq \|pT\|^2 \cdot {\rm id}_{p{\mathcal N}}
\]
holds in the operator algebra ${\rm End}_{pA^{**}}(p{\mathcal N})$,
the centre of which is $*$-isomorphic to $Z(pA^{**})$.
Therefore, the supremum of this increasingly directed bounded net of
positive elements exists as an element of the centre
of the operator algebra ${\rm End}_{pA^{**}}(p{\mathcal N})$, which is
$*$-isomorphic to the von Neumann algebra $Z(pA^{**})$. We denote
the supremum of this net by $\lambda_p$. By construction and by the
w*-continuity of transfers to suprema of increasingly directed bounded
nets of self-adjoint elements of von Neumann algebras we have the
equality
\[
   \lambda_p = {\rm w^*-lim}_{I_0 \subseteq I} \sum_{\alpha \in I_0}
      \lambda_{q_\alpha} \cdot q_\alpha \in Z(pA^{**})
      \equiv Z({\rm End}_{pA^{**}}(p{\mathcal N}))
\]
where $I_0$ runs over the partially ordered net of all finite subsets
of $I$.
Since $\langle q_\alpha pT^*T(z), z \rangle = \lambda_{q_\alpha}^2
q_\alpha \langle z,z \rangle$ for any $z \in q_\alpha {\mathcal N}$
and for any $\alpha \in I$, we arrive at the equality
\[
   \langle pT^*T(z), z \rangle = \lambda_p^2 \cdot p \langle z,z \rangle
\]
for any $z \in p{\mathcal N}$ and for the constructed positive
$\lambda_p \in  Z(pA^{**}) \equiv Z({\rm End}_{pA^{**}}(p{\mathcal N}))$.
Consequently, the operator $pT$ can be written as $pT = \lambda_p V_p$
for some isometric $pA^{**}$-linear map $V_p \in {\rm End}_{pA^{**}}
(p{\mathcal N})$, cf.~\cite[Proposition 2.3]{IT2008}.

Consider the operator $pT$ on $p{\mathcal N}$. Since the formula
\begin{equation} \label{equal999}
   \langle pT(x),pT(x) \rangle = \lambda_p^2 \langle x,x \rangle
   \in \rho \circ \pi(A)
\end{equation}
holds for any $x \in \rho'\circ\pi'({\mathcal M}) \subseteq
p{\mathcal N}$ and since the range of the $A$-valued inner product
on $\mathcal M$ is supposed to be the entire $C^*$-algebra $A$, the
right side of this equality and  the multiplier theory of
$C^*$-algebras forces $\lambda_p^2 \in LM(pA) \cap Z(pA^{**}) =
Z(M(\rho \circ \pi(A))) = \rho \circ \pi (Z(M(A)))$, \cite{Ped1979}.
Taking the square root of $\lambda_p^2$ in a $C^*$-algebraical sense
is an operation which results in a (unique) positive element of the
$C^*$-algebra itself. So we arrive at $\lambda_p \in \rho \circ \pi
(Z(M(A)))$ as the square root of $\lambda_p^2 \geq 0$. In particular,
the operator $\lambda_p \cdot {\rm id}_{p{\mathcal N}}$ preserves the
$\rho \circ \pi(A)$-submodule $\rho' \circ \pi'({\mathcal M})$.

As a consequence, we can lift the bounded $pA^{**}$-linear
orthogonality-preserving operator $pT$ on $p{\mathcal N}$ back to
${\mathcal M}^{\#}$ since $A^{**}$ allows polar decomposition for any
element, the embedding $\rho \circ \pi: A \to pA^{**}$
and the module and operator mappings, induced by $\rho \circ \pi$ and
by Paschke's embedding were isometrically and algebraically, just by
multiplying with or, resp., acting by $p$ in the second step.
So we obtain a decomposition $T = \lambda V$ of $T \in {\rm End}_A
({\mathcal M})$ with a positive function $\lambda \in Z(M(A))
\equiv Z({\rm End}_A({\mathcal M}))$ derived
from $\lambda_p$, and with an isometric $A$-linear embedding $V \in
{\rm End}_A({\mathcal M}^{\#})$, $V$ derived from $V_p$.

In case any element $\lambda' \in Z(M(A)))$ with $|\lambda'|=\lambda$
admits a polar decomposition inside $Z(M(A)))$ then the operator $V$
preserves $\pi'({\mathcal M}) \subset {\mathcal M}^{\#}$. So $T =
\lambda \cdot V$ on ${\mathcal M}$.

For completeness just note, that the  adjointability of $V$ goes
lost on this last step of the proof in case $T$ has not been
adjointable on ${\mathcal M}$ in the very beginning.
\end{proof}

\begin{theorem} \label{thm02}
   Let $A$ be a $C^*$-algebra and $\mathcal M$ be a Hilbert $A$-module.
   Any orthogonality-preserving bounded $A$-linear operator $T$ on
   $\mathcal M$ fulfils the equality
   \[
      \langle T(x),T(y) \rangle = \kappa \langle x,y \rangle
   \]
   for a certain $T$-specific positive element $\kappa \in Z(M(A))$
   and for any $x,y \in \mathcal M$.
\end{theorem}

\begin{proof}
We have only to remark that the values of the $A$-valued inner product
on $\mathcal M$ do not change if $\mathcal M$ is canonically embedded
into ${\mathcal M}^{\#}$ or $\mathcal N$. Then the obtained formula works in the bidual
situation, cf.~(\ref{equal999}).
\end{proof}

\begin{problem}
{\rm We conjecture that any orthogonality-preserving map $T$ on
Hilbert $A$-modules $\mathcal M$ over $C^*$-algebras $A$ are of the
form $T = \lambda V$ for some element $\lambda \in Z((M(A))$ and some
$A$-linear isometry $V: {\mathcal M} \to {\mathcal M}$. To solve this
problem one has possibly to solve the problem of general polar
decomposition of arbitrary elements of (commutative) $C^*$-algebras 
inside corresponding local multiplier algebras or in similarly 
derived algebras.}
\end{problem}

\begin{corollary}
   Let $A$ be a $C^*$-algebra and $\mathcal M$ be a Hilbert $A$-module.
   Let $T$ be an orthogonality-preserving bounded $A$-linear operator on
   $\mathcal M$ of the form $T = \lambda V$, where $V: {\mathcal M}
   \to {\mathcal M}$ is an isometric adjointable bounded $A$-linear
   embedding and $\lambda$ is an element of the centre $Z(M(A))$ of the
   multiplier algebra $M(A)$ of $A$. Then the following conditions
   are equivalent:
   \newcounter{cou001}
   \begin{list}{(\roman{cou001})}{\usecounter{cou001}}
   \item  $T$ is adjointable.
   \item  $V$ is adjointable.
   \item  The graph of the isometric embedding $V$ is a direct
          orthogonal summand of the Hilbert $A$-module
          ${\mathcal M} \oplus {\mathcal M}$.
   \item  The range ${\rm Im}(V)$ of $V$ is a direct orthogonal summand
          of $\mathcal M$.
   \end{list}
\end{corollary}

\begin{proof}
Note, that a multiplication operator by an element $\lambda \in
Z(M(A))$ is always adjointable. So, if $T$ is supposed to be
adjointable, then the operator $V$ has to be adjointable, and
vice versa.
By \cite[Cor.~3.2]{FS2008-1} the bounded operator $V$ is
adjointable if and only if its graph is a direct orthogonal
summand of the Hilbert $A$-module ${\mathcal M} \oplus {\mathcal
M}$. Moreover, since the range of the isometric $A$-linear
embedding $V$ is always closed, adjointability of $V$ forces $V$
to admit a bounded $A$-linear generalized inverse operator on
$\mathcal M$, cf.~\cite[Prop.~3.5]{FS2008-2}. The kernel of this
inverse to $V$ mapping serves as the orthogonal complement of
${\rm Im}(V)$, and ${\mathcal M}= {\rm Im}(V) \oplus {\rm Im}(V)^\bot$
as an orthogonal direct sum by \cite[Th.~3.1]{FS2008-2}.
Conversely, if the range ${\rm Im}(V)$ of $V$ is a direct orthogonal
summand of $\mathcal M$, then there exists an orthogonal projection
of $\mathcal M$ onto this range and, therefore, $V$ is adjointable.
\end{proof}

\section{$C^*$-conformal and conformal mappings}

We want to describe generalized $C^*$-conformal mappings on Hilbert
$C^*$-modules.
A full characterization of such maps involves isometries as
for the orthogonality-preserving case since we resort to a
particular case of the latter.

Let $\mathcal M$ be a Hilbert module over a $C^*$-algebra $A$.
An injective bounded module map $T$ on ${\mathcal M}$ is
said to be {\it $C^*$-conformal} if the identity
\begin{equation}   \label{eq:$C^*$-conformal_mapping_condition}
   \frac{\langle Tx,Ty\rangle}{\|Tx\|\|Ty\|}=\frac{\langle
   x,y\rangle}{\|x\|\|y\|}
\end{equation}
holds for all non-zero vectors $x,y\in {\mathcal M}$. It is
said to be {\it conformal} if the identity
\begin{equation}   \label{eq:conformal_mapping_condition}
   \frac{\|\langle Tx,Ty\rangle\|}{\|Tx\|\|Ty\|}=\frac{\|\langle
   x,y\rangle\|}{\|x\|\|y\|}
\end{equation}
holds for all non-zero vectors $x,y\in {\mathcal M}$.

\begin{theorem} \label{teo:$C^*$-conformal_maps}
  Let $\mathcal M$ be a Hilbert $A$-module over a $C^*$-algebra
  $A$ and $T$ be an injective bounded module map. The following
  conditions are equivalent:
  \begin{list}{(\roman{cou001})}{\usecounter{cou001}}
    \item $T$ is $C^*$-conformal;
    \item $T=\lambda U$ for some non-zero positive $\lambda \in
    \mathbb R$ and for some isometrical module operator
    $U$ on $\mathcal M$.
  \end{list}
\end{theorem}

\begin{proof}
The condition (ii) implies condition (i) because
the condition $\|Ux\|=\|x\|$ for all $x\in M$ implies the
condition $\langle Ux,Uy\rangle=\langle x,y\rangle$ for all $x,
y\in M$ by \cite[Proposition 2.3]{IT2008}. So we have only to
verify the implication (i)$\to$(ii).

Assume an injective bounded module map $T$ on $\mathcal M$
to be $C^*$-conformal. We can rewrite
(\ref{eq:$C^*$-conformal_mapping_condition}) in the following equivalent
form:
\begin{equation}   \label{eq:$C^*$-conformal_mapping_condition2}
   \langle Tx,Ty\rangle=\langle x,y\rangle
   \frac{\|Tx\|\|Ty\|}{\|x\|\|y\|},\quad x,y\neq 0 \, .
\end{equation}
Consider the left part of this equality as a new $A$-valued inner 
product on $\mathcal M$. Consequently, the right part of
(\ref{eq:$C^*$-conformal_mapping_condition2}) has to satisfy all the
conditions of a $C^*$-valued inner product, too. In particular, the
right part of (\ref{eq:$C^*$-conformal_mapping_condition2}) has to be
additive in the second variable, what exactly means
\[
  \langle x,y_1+y_2\rangle\frac{\|Tx\|\|T(y_1+y_2)\|}{\|x\|\|y_1+y_2\|}=
    \langle x,y_1\rangle\frac{\|Tx\|\|Ty_1\|}{\|x\|\|y_1\|} + \langle
    x,y_2\rangle\frac{\|Tx\|\|Ty_2\|}{\|x\|\|y_2\|}
\]
for all non-zero $x,y_1,y_2\in M$. Therefore,
\[
    (y_1+y_2)\frac{\|T(y_1+y_2)\|}{\|y_1+y_2\|}=
    y_1\frac{\|Ty_1\|}{\|y_1\|} + y_2\frac{\|Ty_2\|}{\|y_2\|},
\]
by the arbitrarity of $x \in \mathcal M$, which can be rewritten as
\begin{equation}  \label{eq:additivity1}
    y_1\left(\frac{\|T(y_1+y_2)\|}{\|y_1+y_2\|}-\frac{\|Ty_1\|}{\|y_1\|}\right)+
    y_2\left(\frac{\|T(y_1+y_2)\|}{\|y_1+y_2\|}-\frac{\|Ty_2\|}{\|y_2\|}\right)=0
\end{equation}
for all non-zero $y_1,y_2\in M$. In case the elements $y_1$ and
$y_2$ are not complex multiples of each other both the complex
numbers inside the brackets have to equal to zero. So we arrive at
\begin{equation} \label{eq:conf_map_final}
    \frac{\|T(y_1)\|}{\|y_1\|} = \frac{\|T(y_2)\|}{\|y_2\|}
\end{equation}
for any $y_1, y_2 \in {\mathcal M}$ which are not complex
multiples of one another. Now, if the elements would be
non-trivial complex multiples of each other both the coefficients
would have to be equal, what again forces equality
(\ref{eq:conf_map_final}).

Let us denote the positive real number
$\frac{\|Tx\|}{\|x\|}$ by $t$. Then the equality
(\ref{eq:conf_map_final}) provides
\[
\left\|\left(\frac 1t T\right)(z)\right\|=\|z\|,
\]
which means $U=\frac{1}{t} T$ is an isometrical operator.
The proof is complete.
\end{proof}

\begin{example}
{\rm Let $A=C_0((0,1])={\mathcal M}$ and $T$ be a $C^*$-conformal
mapping on $\mathcal M$. Our aim is to demonstrate that $T=t U$
for some non-zero positive $t \in \mathbb{R}$ and for some
isometrical module operator $U$ on $\mathcal M$. To begin with,
let us recall that the Banach algebra ${\rm End}_{A}(M)$ of all
bounded module maps on $\mathcal M$ is isomorphic to the algebra
$LM(A)$ of left multipliers of $A$ under the given circumstances.
Moreover, $LM(A)=C_b((0,1])$, the $C^*$-algebra of all
bounded continuous functions on $(0,1]$. So any $A$-linear bounded
operator on $\mathcal M$ is just a multiplication by a certain
function of $C_b((0,1])$. In particular,
\[
T(g)=f_T \cdot g, \quad g\in A,
\]
for some $f_T\in C_b((0,1])$. Let us denote by $x_0$ the point of
$(0,1]$, where the function $|f_T|$ achieves its supremum, i.e.
$|f_T(x_0)|=\|f_T\|$, and set $t:=\|f_T\|$. We claim that the
operator $\frac{1}{t} T$ is an isometry, what exactly means
\begin{equation}    \label{eq:ex_conf_map_C_0[0,1]}
   \frac{|f_T(x)|}{\|f_T\|}=1
\end{equation}
for all $x\in (0,1]$. Indeed, consider any point $x\neq x_0$.
Let $\theta_x\in C_0((0,1])$ be an Urysohn function for
$x$, i.e. $0\le \theta_x\le 1$, $\theta_x(x)=1$ and $\theta_x=0$
outside of some neighborhood of $x$, and let $\theta_{x_0}$ be a
Urysohn function for $x_0$. Moreover, we can assume that the
supports of $\theta_{x}$ and $\theta_{x_0}$ do not intersect each
other. Now the condition (\ref{eq:$C^*$-conformal_mapping_condition})
written for $T$ and for coinciding vectors $x=y=\theta_{x}+\theta_{x_0}$
yields the equality
\[
  \frac{|f_T|^2(\theta_{x}+\theta_{x_0})^2}{\|f_T(\theta_{x}+\theta_{x_0})\|^2}=
  \frac{(\theta_{x}+\theta_{x_0})^2}{\|\theta_{x}+\theta_{x_0}\|^2},
\]
which implies
\[
  \frac{|f_T|^2(\theta_{x}+\theta_{x_0})^2}{\|f_T\|^2}=
  (\theta_{x}+\theta_{x_0})^2.
\]
This equality at point $x$ takes the form
(\ref{eq:ex_conf_map_C_0[0,1]}) for any $x\in (0,1]$. }
\end{example}

\begin{theorem} \label{teo:conformal_maps}
  Let $\mathcal M$ be a Hilbert $A$-module over a $C^*$-algebra
  $A$ and $T$ be an injective bounded module map. The following
  conditions are equivalent:
  \begin{list}{(\roman{cou001})}{\usecounter{cou001}}
    \item $T$ is conformal;
    \item $T=\lambda U$ for some non-zero positive $\lambda \in
    \mathbb R$ and for some isometrical module operator
    $U$ on $\mathcal M$.
  \end{list}
\end{theorem}

\begin{proof}
As in the proof of the theorem on orthogonality-preserving mappings we
switch from the setting $\{ A, {\mathcal M}, T \}$ to its faithful
isometric representation in $\{ pA^{**}, p{\mathcal M}^{\#} \subseteq
p{\mathcal N}, T \}$, where $p \in A^{**}$ is the central projection
of $A^{**}$ mapping $A^{**}$ to its discrete part.

First, consider a minimal projections $e \in pA^{**}$. Then the
equality (\ref{eq:conformal_mapping_condition}) gives
\[
   \frac{\|\langle ex,ey \rangle\|}{\|ex\| \|ey\|} =
   \frac{\|\langle T(ex),T(ey) \rangle\|}{\|T(ex)\| \|T(ey)\|}
\]
for any $x,y \in p{\mathcal M}^{\#}$. Since $\{ e{\mathcal M}^{\#},
\langle .,. \rangle \}$ becomes a Hilbert space after factorization
by the set $\{ x \in p{\mathcal M}^{\#}: e\langle x,x \rangle e =0 \}$,
the map $T$ acts as a positive scalar multiple of a linear isometry
on $e{\mathcal M}^{\#}$, i.e. $eT = \lambda_e U_e$.

Secondly, every two minimal projections $e,f \in pA^{**}$ with the
same minimal central support projection $q \in pZ(A^{**})$ are connected by
a (unique) partial isometry $u \in pA^{**}$ such that $u^*u=f$ and $uu^*=e$.
Arguments analogous to those given at \cite[p.~303]{IT2008} show
\begin{eqnarray*}
\lambda_e^2 \cdot e\langle x,x \rangle e & = & ufu^* \langle T(x),T(x) \rangle ufu^*\\
   & = & uf \langle T(u^*x)t(u^*x) \rangle fu^* \\
   & = & u \lambda_f^2 f \langle u^*x,u^*x \rangle f \\
   & = & \lambda_f^2 \cdot e \langle x,x \rangle e \, .
\end{eqnarray*}
Therefore, $qT = \lambda_q U$ for some positive $\lambda_q \in \mathbb R$,
for a $qA$-linear isometric mapping $U: q{\mathcal M}^{\#} \to q{\mathcal M}^{\#}$
and for any minimal central projection $q \in pA^{**}$.

Thirdly, suppose $e,f$ are two minimal central projections of $pA^{**}$
that are orthogonal.
For any $x,y \in p{\mathcal M}^{\#}$ consider the supposed equality
\[
\frac{\|\langle (e+f)x,(e+f)y \rangle \|}{\|(e+f)x\| \|(e+f)y\|} =
\frac{\|\langle T((e+f)x),T((e+f)y) \rangle \|}{\|T((e+f)x)\| \|T((e+f)y)\|} \, .
\]
Since $T$ is a bounded module mapping which acts on $ep{\mathcal M}^{\#}$
like $\lambda_e \cdot {\rm id}$ and on $fp{\mathcal M}^{\#}$
like $\lambda_f \cdot {\rm id}$ we arrive at the equality
\[
\frac{\|\langle (e+f)x,(e+f)y \rangle \|}{\|(e+f)x\| \|(e+f)y\|} =
\frac{\|\langle (\lambda_e e + \lambda_f f)x, (\lambda_e e + \lambda_f f)y \rangle \|}{\|(\lambda_e e + \lambda_f f)x\| \|(\lambda_e e + \lambda_f f)y\|} \, .
\]
Involving the properties of $e,f$ to be central and orthogonal to
each other and exploiting modular linear properties of the $pA^{**}$-valued
inner product we transform the equality further to
\begin{eqnarray*}
\lefteqn{  \frac{\|\langle ex,ey \rangle+\langle fx,fy \rangle\|}{\|\langle ex,ex \rangle+\langle fx,fx \rangle\|^{1/2} \cdot \|\langle ey,ey \rangle+\langle fy,fy \rangle\|^{1/2}} = } \\
& = & \frac{\|\lambda_e^2 \langle ex,ey \rangle + \lambda_f^2 \langle fx,fy \rangle\|}{\|\lambda_e^2 \langle ex,ex \rangle + \lambda_f^2 \langle fx,fx \rangle\|^{1/2} \cdot
\|\lambda_e^2 \langle ey,ey \rangle + \lambda_f^2 \langle fy,fy \rangle\|^{1/2}} \, .
\end{eqnarray*}
Since $e,f$ are pairwise orthogonal central projections we can
transform the equality further to
\begin{eqnarray*}
\lefteqn{ \frac{\sup \{\|\langle ex,ey \rangle\|, \|\langle fx,fy \rangle\|\}}{\sup \{\|\langle ex,ex \rangle\|^{1/2},\|\langle fx,fx \rangle\|^{1/2}\} \cdot
\sup \{\|\langle ey,ey \rangle\|^{1/2}, \|\langle fy,fy \rangle\|^{1/2}\}} = } \\
& = & \frac{\sup \{\|\lambda_e^2 \langle ex,ey \rangle\|, \|\lambda_f^2 \langle fx,fy \rangle\|\}}{\sup \{ \|\lambda_e^2 \langle ex,ex \rangle\|^{1/2},
\|\lambda_f^2 \langle fx,fx \rangle\|^{1/2} \} \cdot \sup \{\|\lambda_e^2 \langle ey,ey \rangle\|^{1/2}, \|\lambda_f^2 \langle fy,fy \rangle\|^{1/2}\}} \, .
\end{eqnarray*}
By the w*-density of $p{\mathcal M}^{\#}$ in $p{\mathcal N}$ we can
distinguish a finite number of cases at which of the central parts
the respective six suprema may be admitted, at $epA^{**}$ or at $fpA^{**}$.
For this aim we may assume, in particular, that $x,y$ belong to $p{\mathcal N}$ to have a larger
set for these elements to be selected specifically. Most interesting
are the cases when (i) both $(e+f)x$ and $(\lambda_e e + \lambda_f f)x$
admit their norm at the $e$-part, (ii) both $(e+f)y$ and $(\lambda_e e +
\lambda_f f)y$ admit their norm at the $f$-part, and (iii) both
$\langle (e+f)x,(e+f)y\rangle$ and $\langle (\lambda_e e + \lambda_f f)x,
(\lambda_e e + \lambda_f f)y \rangle$ admit their norm (either) at
the $e$-part (or at the $f$-part). In these cases the equality above
gives $\lambda_e=\lambda_f$. (All the other cases either give the same result
or do not give any new information on the interrelation of $\lambda_e$
and $\lambda_f$.)

Finally, if for any central minimal projection $f \in pA^{**}$
the operator $T$ acts on $fp{\mathcal N}$ as $\lambda U$ for a certain
(fixed) positive constant $\lambda$ and a certain module-linear isometry $U$
then $T$ acts on $p{\mathcal N}$ in the same way. Consequently,
$T$ acts on $\mathcal M$ in the same manner since $U$ preserves
$p{\mathcal M}^{\#}$ inside $p{\mathcal N}$.
\end{proof}

\begin{remark} {\rm
Obviously, the $C^*$-conformity of a bounded module map follows from the
conformity of it, but the converse is not obvious, even it is
true for Hilbert $C^*$-modules. }
\end{remark}

\medskip \noindent
{\bf Acknowledgements:} We are grateful to Chi-Keung Ng who pointed
us to the results by G.~K.~Pedersen in September 2009. So we had to correct 
a crucial argument in the second paragraph of Thm. \ref{main} giving other
and much more detailled arguments.



\end{document}